\documentclass[11pt,A4paper]{article}

\usepackage[left=30mm,right=30mm,top=25mm,bottom=30mm]{geometry}
\usepackage{labelfig}
\usepackage{epsfig}
\usepackage{epstopdf}
\usepackage{soul}
\usepackage{xfrac}

\usepackage[utf8]{inputenc}
\usepackage[T1]{fontenc}
\usepackage[main=english,french]{babel}

\usepackage[percent]{overpic}

\usepackage{color}
\usepackage{amsthm,amsmath,amssymb}
\usepackage{booktabs}
\usepackage{mathpazo}
\usepackage{microtype}
\usepackage{overpic}
\usepackage{bm}
\usepackage{sectsty}
\usepackage[
	pdftitle={PDFTitle},
	pdfauthor={Hugo Parlier},
	ocgcolorlinks,
	linkcolor=linkred,
	citecolor=linkred,
	urlcolor=linkblue]
{hyperref}

\definecolor{linkred}{RGB}{148,33,147} %DeepSkyBlue
\definecolor{linkblue}{RGB}{16, 78, 139}

\usepackage[hang,flushmargin]{footmisc}
\usepackage{enumitem}
\usepackage{titlesec}
	\titlespacing{\section}{0pt}{12pt}{0pt}
	\titlespacing{\subsection}{0pt}{6pt}{0pt}

\newcommand{\plainfootnote}[1]{%
  \begingroup
    \renewcommand{\thefootnote}{}%
    \footnotetext{#1}%
    \addtocounter{footnote}{-1}%
  \endgroup
}
	
\titlelabel{\thetitle.\quad}

\makeatletter 

\long\def\@footnotetext#1{% 
\H@@footnotetext{% 
\ifHy@nesting 
\hyper@@anchor{\@currentHref}{#1}% 
\else 
\Hy@raisedlink{\hyper@@anchor{\@currentHref}{\relax}}#1% 
\fi 
}}

\def\@footnotemark{% 
\leavevmode 
\ifhmode\edef\@x@sf{\the\spacefactor}\nobreak\fi 
\H@refstepcounter{Hfootnote}% 
\hyper@makecurrent{Hfootnote}% 
\hyper@linkstart{link}{\@currentHref}% 
\@makefnmark 
\hyper@linkend 
\ifhmode\spacefactor\@x@sf\fi 
\relax 
}% 

\makeatother 

\theoremstyle{plain}
\newtheorem{theorem}{Theorem}[section]
\newtheorem{proposition}[theorem]{Proposition}
\newtheorem{lemma}[theorem]{Lemma}

\theoremstyle{definition}

\newcommand{\Hyp}{{\mathbb H}}

\newcommand{\Z}{{\mathbb Z}}

\newcommand{\M}{{\mathcal M}}
\newcommand{\PSL}{{\rm PSL}}

\newcommand{\arcsinh}{{\,\rm arcsinh}}
\newcommand{\arccosh}{{\,\rm arccosh}}

\sectionfont{\large \bfseries}
\subsectionfont{\normalsize}

\long\def\symbolfootnote[#1]#2{\begingroup%
\def\thefootnote{\fnsymbol{footnote}}\footnote[#1]{#2}\endgroup}

\def\blfootnote{\xdef\@thefnmark{}\@footnotetext}

\usepackage{mathtools}

\usepackage[sort,nocompress]{cite}

\begin{document}

{\Large \bfseries Near ideal decompositions of ideal polygons}

{\large
Hugo Parlier \par}

\plainfootnote{%
  }

\vspace{0.5cm}
{\bf Abstract.}
This article gives a short proof that all ideal polygons admit a short orthogeodesic decomposition. Specifically, all $n$-gons admit an orthogeodesic decomposition with orthogeodesics all of length at most $\sim 2 \log(n)$, and this is roughly optimal. 
\vspace{0.5cm}

\section{Introduction} \label{sec:intro}

In hindsight, the study of moduli spaces of hyperbolic surfaces began with Riemann's moduli problem, asking for a "good" parameter set for conformal classes of surfaces. One version is to find an exact classification of conformal classes, which, taking into account the uniformization theorem, is the question of finding a precise description of the moduli space of hyperbolic metrics on an orientable closed surface of genus $g\geq 2$. While this is difficult in general, the question of finding rough parametrizations, where rough can take different meanings, has been solved in different ways. 

One interpretation is the search for a fundamental domain for the action of the mapping class group on Teichmüller space, a problem going back to the 1970s, see for instance \cite{Keen}. By considering Fenchel-Nielsen coordinates and a theorem of Bers which states that any surface can be decomposed into pants decompositions of length at most a constant that only depends on the topology, you can find a rough fundamental domain, where each conformal class of surface appears at most a fixed number of times. This leads to questions about the growth of so-called Bers constants in terms of the genus, with lower bounds on the order of $\sqrt{g}$ and upper bounds on the order of $g$ (see \cite{Bers, BuserBook, BalacheffParlier, BalacheffParlierSabourau, ParlierShortBers, ParlierShorterBers}). 

Here a similar problem is considered, but this time for ideal hyperbolic $n$-gons, whose moduli space will be denoted by $\M_n$. An ideal $n$-gon can be cut along maximal collections of orthogeodesics, that is geodesics orthogonal to the sides of the polygon, pairwise disjoint and maximal with respect to inclusion. These will be called orthogeodesic decompositions, and are natural replacements for pants decompositions in this context. In fact, an ideal polygon can be doubled along its perimeter to obtain an $n$-punctured sphere with a natural reflexion, and the doubling of such an orthogeodesic decomposition results in a pants decomposition, by construction invariant under the reflexion.

The main observation of this note is that a similar result to Bers' holds, and whose asymptotic growth is straightforward to obtain.\\

\begin{theorem}\label{thm:main}
For any $n$, there exists a constant $O_n$ such that any $P \in \M_n$ admits an orthogeodesic decomposition $\mu_1, \hdots, \mu_{n-3}$ where each orthogeodesic is of length at most $O_n$. Furthermore
$$
 2 \arcsinh\left( \frac{3}{2} \cot\left(\frac{\pi}{n} \right) \right)  \leq O_n \leq 2 \arccosh\left( \frac{1}{\sin\left(\frac{\pi}{n} \right)}\right)
$$ 
and in particular
$$
\lim_{n\to \infty} \sfrac{O_n}{2 \log(n)} =1.
$$
\end{theorem}

One of the main ingredients in the proof is the use of the upper bound on the size of the largest disk on an ideal polygon (Proposition \ref{prop:inradius}), which is an ideal polygon version of a result of Bavard \cite{Bavard} for closed surfaces.

The moduli space $\M_n$ has been studied before, namely in \cite{ParkerEtAl} where billard paths are studied. The authors observe that this moduli space can be thought of as the subset of the moduli space of $n$-punctured spheres with an orientation reversing involution that fixes the punctures. Likewise, you can double the ideal polygon to get a punctured sphere. Now within the larger Teichmüller space of the punctured sphere, this forms a geodesically convex subspace for the Weil-Petersson metric, and the authors use the convexity of length functions \cite{Wolpert} to prove that the average billard path length with fixed number of bounces is uniquely minimized for the polygon $P_n$ above.

{\bf Organization.} The article is organized as follows. After a preliminary and setup section,  Section \ref{sec:short} is dedicated to proving the upper bound of Theorem \ref{thm:main} while in Section \ref{sec:long}, the lower bound is provided. 

{\bf Acknowledgements.} This paper comes from using the space of ideal polygons as a toy moduli space in various doctoral courses, including in Strasbourg in 2017, Thailand in 2019 and Bern in 2025.  The author is very grateful to the organizers and the participants for these wonderful opportunities to share ideas and draw pictures for eager audiences. Many thanks to Marie Abadie for commenting on an earlier version of this draft, and to the referee for comments which helped improve exposition in certain crucial places. The author is supported by ANR-SNF Grant number 200021E\_238147 (SUGAR). 
\section{Preliminaries and setup}

An ideal $n$-gon in the hyperbolic plane $\Hyp$ is a polygon with its $n$ vertices on the boundary of $\Hyp$. The space of all ideal $n$-gons up to isometry is a moduli space, denoted here by $\M_n$. For $n=3$, the moduli space consists of one point, as there is only one ideal triangle up to isometry, so from now on $n\geq 4$. As $2n-3$ lengths and angles determine a hyperbolic polygon, and $n$ angles are equal to $0$, $\M_n$ is of real dimension $n-3$. 

Given two non-adjacent sides of $P\in \M_n$, there is a unique geodesic between them which realizes their distance, and orthogonal in both endpoints. These will be called orthogeodesics. Note that any path between the two corresponding sides is homotopic, with endpoints allowed to move along the sides, to this unique orthogeodesic. This useful fact can be thought of as the "ideal polygon" version of the unicity of closed geodesics in the free homotopy class of a closed curve on a hyperbolic surface. By a counting argument, any ideal $n$-gon has exactly $\frac{1}{2} n (n-2)$ orthogeodesics. Note that on surfaces, also with boundary but with more topology, there are infinitely many orthogeodesics which have been studied in different contexts. For instance, Basmajian and Fanoni studied properties of the shortest orthogeodesic \cite{Basmajian-Fanoni}. 

A set of disjoint orthogeodesics which is maximal with respect to inclusion is called an orthogeodesic decomposition. An orthogeodesic decomposition is always a collection of $n-3$ orthogeodesics. There are many ways to see this. One way is to associate it to a triangulation, by orienting the boundary of the polygon and pulling each endpoint of the orthogeodesic to a vertex, following the orientation. A triangulation of an $n$-gon has $n-3$ edges, so the orthogeodesic decomposition had $n-3$ elements as well. Another way is to observe that the complementary region of the decomposition consists in a collection of pieces that are either right-angled hexagons, right-angled pentagons with one ideal point, or a quadrilateral with two adjacent ideal points and two right angles (see Figure \ref{fig:pieces}). These pieces will be called {\it elementary}. 

\begin{figure}[h]
%\ShowGrid
\leavevmode \SetLabels
%\L(.51*.78) $p$\\%
%\L(.5*.42) $q$\\%
%\L(.505*.54) $\delta$\\%
%\L(.49*.00) $\alpha$\\%
\endSetLabels
\begin{center}
\AffixLabels{\centerline{\includegraphics[width=12cm]{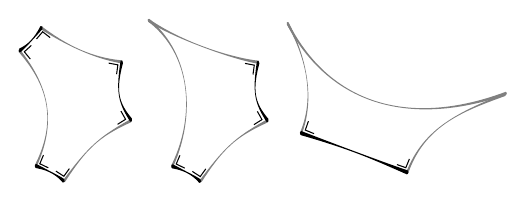}}}
\vspace{-24pt}
\end{center}
\caption{The different types of complementary regions to an orthogeodesic decomposition} \label{fig:pieces}
\end{figure}

Observe that one can think of the elementary pieces as generalized right-angled hexagons, where one or two of the sides are allowed to be of length $0$. In particular, the lengths of the orthogeodesic sides determine their geometry and they are all area $\pi$. Unlike for pants decompositions, there are no twist parameters, as the endpoints of the orthogeodesics have to match up. In particular this means:
\begin{proposition}
The $n-3$ (marked) lengths of the orthogeodesics in an orthogeodesic decomposition determine $P\in \M_n$ uniquely. 
\end{proposition}

Note that for hyperbolic trigonometry formulas, which will be used throughout this note, one can refer to \cite{BuserBook}, page 454.

To end this preliminary section, the case where $n=4$ is considered. Note that in this case an orthogeodesic decomposition consists in a single orthogeodesic. 

\begin{proposition}
Any $P \in \M_4$ admits an orthogeodesic of length at most $2\arcsinh(1)$, and the bound is sharp. In particular $\M_4$ is in one-to-one correspondence with the interval $]0,2\arcsinh(1)]$ where $t \in ]0,2\arcsinh(1)]$ represents the length of the shortest orthogeodesic of the ideal $4$-gon.
\end{proposition}
\begin{proof}
An ideal $4$-gon only has 2 orthogeodesics, say of lengths $x$ and $y$ with $x\leq y$, and it is not difficult to see that they intersect orthogonally and bisect each other. They split the ideal quadrilateral into 4 isometric so-called Lambert quadrilaterals, with 3 right angles and an ideal point. The non-infinite lengths of these smaller quadrilaterals are thus $\frac{x}{2}$ and $\frac{y}{2}$. By a standard formula for quadrilaterals, 
$$
\sinh\left(\frac{x}{2}\right)\sinh\left(\frac{y}{2}\right)=1
$$
from which we can deduce that $x \leq 2 \arcsinh(1)$ and $y \geq 2 \arcsinh(1)$. This proves the claim. 
\end{proof}

Note that there is a unique quadrilateral with $2$ equal orthogeodesics (of lengths $2 \arcsinh(1)$) and which has a rotational symmetry. This rotational symmetry means that this quadrilateral will correspond to an orbifold point in the $1$-dimensional moduli space underlying moduli space. In fact, it is the unique orbifold point in the moduli space, and thus plays a similar role to the square and the hexagonal torus in the moduli space of flat tori which are the orbifold points of the modular surface $\Hyp / \PSL_2(\Z)$. 

\section{Embedded disks and short orthogeodesic decompositions}\label{sec:short}

One of the more remarkable members  of $ \M_n $ is the unique regular ideal $n$-gon $P_n$, which can be constructed, in the Poincaré disk model, by taking as vertices $n$ evenly spaced points on the boundary circle of the hyperbolic plane. It has a maximally embedded disk, tangent to each of its $n$ sides. We begin by computing the radius of this disk, which we think of as the maximal inner radius of $P_n$. 

\begin{lemma}
The largest embedded disk in $P_n$ has radius $r_n$ where 
$$
r_n= \arccosh\left( \frac{1}{\sin\left(\frac{\pi}{n} \right)}\right)
$$
\end{lemma}

\begin{proof}
Using the symmetry of the polygon, we can reduce the computation to a single right-angled triangle with one ideal point, and the unique non-infinite side of length $r_n$. Following a standard hyperbolic trigonometry formula (see Figure \ref{fig:triangle}), we have 
$$
\cos(0) = \cosh(r_n) \sin\left(\frac{\pi}{n} \right)
$$
from which the formula is easily deduced.
\end{proof}

\begin{figure}[h]
%\ShowGrid
\leavevmode \SetLabels
%\L(.51*.78) $p$\\%
%\L(.5*.42) $q$\\%
\L(.7*.54) $r_n$\\%
\L(.73*.47) $\frac{\pi}{n}$\\%
\endSetLabels
\begin{center}
\AffixLabels{\centerline{\includegraphics[width=12cm]{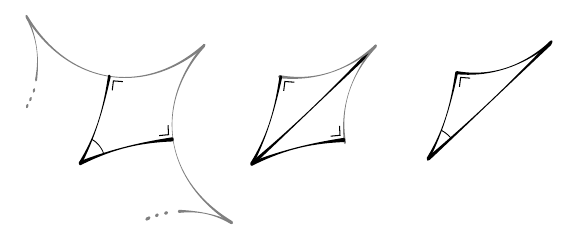}}}
\vspace{-24pt}
\end{center}
\caption{Computing the inradius of $P_n$} \label{fig:triangle}
\end{figure}

For completeness, we provide a quick proof of the following result which is in fact a very simple version of an identity of Basmajian \cite{Basmajian}. 

\begin{lemma}
For any point $p\in P$, $P\in \M_n$, let $x_1,\hdots, x_n$ be the distances to the $n$ sides of $P$. Then 
$$
\sum_{k=1}^n \arcsin \left( \frac{1}{\cosh\left(x_k\right)}\right) = \pi
$$
\end{lemma}
\begin{proof}
The proof is very similar to the previous computation. To each boundary component, associate the triangle it forms together with $p$. The path of length $x_k$ is the height of this triangle with two ideal points, and denote by $\theta_k$ the angle at $p$. By the same formula as before, we can compute $\theta_k$ is function of $x_k$:
$$
\theta_k = 2 \arcsin \left( \frac{1}{\cosh\left(x_k\right)}\right).
$$
Now the angles add up to $2\pi$, so we have 
$$
\sum_{k=1}^n \theta_k = \sum_{k=1}^n  2\arcsin \left( \frac{1}{\cosh\left(x_k\right)}\right) = 2\pi
$$
and the identity follows by dividing by $2$.
\end{proof}

Our final result before proving the theorem is the following corollary of the above result. 

\begin{proposition}\label{prop:inradius}
For $P\in \M_n$ and $p\in P$, the inradius of $p$ is bounded above by 
$$
r_n= \arccosh \left( \frac{1}{\sin\left(\frac{\pi}{n} \right)}\right).
$$
\end{proposition}

\begin{proof} 
Choose any $p \in P$, and consider $x_1,\hdots, x_n$ and $\theta_1, \hdots,\theta_n$ as in the previous lemma. Note that they each satisfy the following equality, obtained by expressing the distance in terms of the angle:
$$
x_k =  \arccosh\left( \frac{1}{\sin\left(\frac{\theta_k}{2} \right)}\right).
$$

Observe that the function $\arcsin \left( \frac{1}{\cos\left(x\right)}\right)$ is decreasing in $x$, hence the maximum angle 
$$\theta_{\max} := \max \{\theta_k \mid k= 1,\hdots,n \}
$$
corresponds to the minimal distance
$$
x_{\min}  := \min \{x_k \mid k= 1,\hdots,n \}
$$
which is also the maximal inradius at $p$. 

Now as $\theta_{\max} \geq \frac{2 \pi}{n}$, we have 
$$
x_{\min} = \arccosh\left( \frac{1}{\sin\left(\frac{\theta_{\min}}{2} \right)}\right) \leq \arccosh\left( \frac{1}{\sin\left(\frac{\pi}{n} \right)}\right)
$$
as claimed. 
\end{proof}
We note that this result is analogous to the (more difficult) result for closed surfaces, namely Bavard's \cite{Bavard} sharp upper bound on the radius of a disk on a closed orientable hyperbolic surface. 

The proof shows that equality only occurs for $P_n$ with $p$ being the central point of the polygon. 

We can now prove the upper bound on the length of orthogeodesics. 

\begin{theorem}\label{thm:upper}
Any $P \in \M_n$ admits an orthogeodesics decomposition $\mu_1, \hdots, \mu_{n-3}$ with 
$$
\ell(\mu_k) \leq 2 \, r_n
$$
for all $k\in \{1,\hdots, n-3\}$. 
\end{theorem}
\begin{proof}

By Proposition~\ref{prop:inradius}, each point $p\in P$ is at distance at most $r_n$ from the boundary of $P$. We decompose $P$ into cells as follows. To each point $p\in P$, we associate the side of $P$ to which it is closest.

This decomposes $P$ into cells whose boundary points correspond to points at equal distance from at least two sides. The boundary of each cell is piecewise geodesic, and the points at which the geodesic is broken correspond to points at equal distance from at least three sides. The union of these boundary points will be referred to as the cut locus. Observe that the cut locus is a geodesic tree, since it cannot contain any cycles. Its vertices have degree at least $3$ (and for a generic polygon $P\in \mathcal M_n$, all vertices have degree exactly $3$, although this fact will not be used).

Now let $x$ be a point lying in the interior of an edge of the cut locus, i.e., not a vertex and not a leaf. By definition, $x$ has exactly two closest sides of $P$, say $s_1$ and $s_2$, and these sides are necessarily non-adjacent. There are exactly two minimizing geodesic segments from $x$ to $\partial P$, one meeting $s_1$ and one meeting $s_2$. Both have equal length, at most $r_n$, and meet $\partial P$ orthogonally. Concatenating these two segments yields a geodesic arc with endpoints on $s_1$ and $s_2$, freely homotopic (rel.\ boundary) to a unique orthogeodesic in $P$. For any two points lying on the same open edge of the cut locus, this construction yields the same homotopy class, and hence the same orthogeodesic. We denote by $\Gamma$ the collection of orthogeodesics obtained in this way, one for each open edge of the cut locus corresponding to a pair of non-adjacent sides. Cut-locus edges corresponding to pairs of adjacent sides occur only near the ideal vertices of $P$; these edges are discarded from the construction, since the associated minimizing segments run into the same ideal vertex and do not determine an orthogeodesic contained in $P$.

If every vertex of the cut locus has degree exactly $3$, then the orthogeodesics in $\Gamma$ already form an orthogeodesic decomposition of $P$, because each complementary region is elementary.  When the cut locus has vertices of degree greater than $3$, the complementary regions bounded by $\Gamma$ are no longer elementary and must be further subdivided. We now describe a local construction which, for each vertex of degree $m\ge 4$, introduces finitely many additional orthogeodesics so as to complete the decomposition near that vertex.

To do this, label the $m\geq 4$ geodesic paths $a_1, \hdots, a_m$ from the vertex of the locus to the boundary by choosing $a_1$ arbitrarily and then by choosing an orientation around the vertex. Orient $a_1$ from the boundary to the vertex and orient the others from the vertex to the boundary. For $k=2, \hdots, m$, choose the $m-1$ homotopy classes $\alpha_k$ of arcs obtained by concatenating $a_1$ and $a_k$ (see Figure \ref{fig:cut}). 

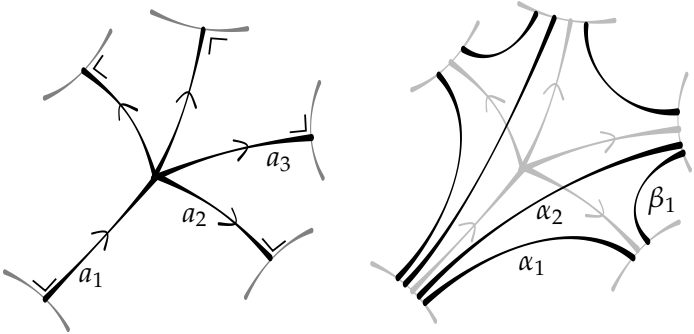
\begin{figure}[h]
%\ShowGrid
\leavevmode \SetLabels
%\L(.51*.78) $p$\\%
%\L(.5*.42) $q$\\%
\L(.255*.24) $a_1$\\%
\L(.343*.393) $a_2$\\%
\L(.415*.535) $a_3$\\%
\L(.63*.27) $\alpha_1$\\%
\L(.645*.4) $\alpha_2$\\%
\L(.74*.43) $\beta_1$\\%
\endSetLabels
\begin{center}
\AffixLabels{\centerline{\includegraphics[width=10cm]{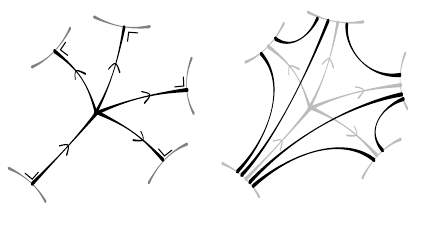}}}
\vspace{-24pt}
\end{center}
\caption{Constructing a decomposition} \label{fig:cut}
\end{figure}

Note that $\alpha_k$ has a unique geodesic minimizer in its homotopy class (with endpoints gliding on the side of the polygon), which is either an orthogeodesic or possibly, for $k=2$ and/or $k=m$, an ideal point. In any case its length is always bounded above by $2 r_n$. To these we add the orthogeodesics $\beta_k$ homotopic to the concatenation of $a_k^{-1}$ and $a_{k+1}$. 

As all of the complementary pieces are elementary, the full collection of orthogeodesics consists in an orthogeodesic decomposition as required.
\end{proof}

\section{A lower bound on $O_n$}\label{sec:long}
In the previous section, an upper bound on $O_n$ was obtained, and so to conclude the proof of Theorem \ref{thm:main}, we need to compute the lower bound. (The statement about asymptotic growth is a straightforward consequence of the two bounds.)

A lower bound comes from looking at $P_n$. We begin by computing the lengths of all of the orthogeodesics of $P_n$. 

Note that any orthogeodesic splits the vertices of $P_n$ into two sets, those lying on either side of the orthogeodesic. If the cardinalities of the two sets are $n_1$ and $n_2$, observe that $n_1, n_2 \geq 2$ and $n_1+n_2=n$. The length of the orthogeodesic only depends on this splitting. If we suppose that $n_1\leq n_2$, we can compute the length of the orthogeodesic observing that is homotopic to the concatenation of two geodesic arcs of length $r_n$ which meet at an angle of $\frac{n_1}{n} 2 \pi$. It can thus be computed using the formula for a quadrilateral with three right angles and one angle equal to $\frac{n_1}{n}  \pi$ (see Figure \ref{fig:shape}). The length of the orthogeodesic is denoted $\ell_{n_1}$ as it only depends on $n_1$. 

\begin{figure}[h]
%\ShowGrid
\leavevmode \SetLabels
%\L(.51*.78) $p$\\%
%\L(.5*.42) $q$\\%
\L(.63*.6) $r_n$\\%
\L(.53*.73) $\frac{n_1}{n}  \pi$\\%
\endSetLabels
\begin{center}
\AffixLabels{\centerline{\includegraphics[width=8cm]{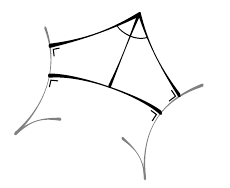}}}
\vspace{-24pt}
\end{center}
\caption{Computing orthogeodesic lengths on $P_n$} \label{fig:shape}
\end{figure}

Using hyperbolic trigonometry again, we have 

$$
\sinh \left( \frac{\ell_{n_1}}{2}\right) = \sin\left( \frac{n_1}{n}\pi \right) \sinh(r_n)
$$
and so 
$$
\ell_{n_1} = 2 \arcsinh\left( \sin\left( \frac{n_1}{n}\pi \right) \sinh(r_n) \right) .
$$

A crucial observation is that any orthogeodesic decomposition of an ideal $n$-gon contains at least one orthogeodesic which separates the polygon into two sub-polygons, each containing at least $\frac{n}{3}$ ideal vertices. Indeed, an orthogeodesic decomposition determines a planar tree whose leaves correspond to the ideal vertices of $P$ and whose internal vertices have degree at most $3$. Any such tree with $n$ leaves contains an edge whose removal separates the tree into two components, each containing at least $\frac{n}{3}$ leaves. This follows from a standard balancing argument: starting from any edge and repeatedly moving across it toward the side containing more than $\frac{2n}{3}$ leaves must eventually terminate at an edge for which neither side contains more than $\frac{2n}{3}$ leaves.

In particular, this means there is at least one orthogeodesic of length 
$$
2 \arcsinh\left( \sin\left( \frac{\pi}{3}  \right) \sinh(r_n) \right) = 2 \arcsinh\left( \frac{3}{2} \cot\left(\frac{\pi}{n} \right) \right).
$$
This proves the lower bound on $O_n$ in Theorem \ref{thm:main}.

{\it Address and email:}\\
Department of Mathematics, University of Fribourg, Switzerland\\
hugo.parlier@unifr.ch


\begin{thebibliography}{99}

\bibitem{BalacheffParlier}
Florent Balacheff and Hugo Parlier.
\newblock Bers' constants for punctured spheres and hyperelliptic surfaces.
\newblock {\em J. Topol. Anal.}, 4(3):271--296, 2012.

\bibitem{BalacheffParlierSabourau}
Florent Balacheff, Hugo Parlier, and Stéphane Sabourau.
\newblock Short loop decompositions of surfaces and the geometry of Jacobians.
\newblock {\em Geom. Funct. Anal.}, 22(1):37--73, 2012.

\bibitem{Basmajian}
Ara Basmajian.
\newblock The orthogonal spectrum of a hyperbolic manifold.
\newblock {\em Amer. J. Math.}, 115(5):1139--1159, 1993.

\bibitem{Basmajian-Fanoni}
Ara Basmajian and Federica Fanoni.
\newblock Orthosystoles and orthokissing numbers.
\newblock {\em Algebr. Geom. Topol.}, to appear.

\bibitem{Bavard}
Christophe Bavard.
\newblock Disques extr\'emaux et surfaces modulaires.
\newblock {\em Ann. Fac. Sci. Toulouse Math. (6)}, 5(2):191--202, 1996.

\bibitem{Bers}
Lipman Bers.
\newblock An inequality for Riemann surfaces.
\newblock In \emph{Differential geometry and complex analysis}, pages 87--93.
Springer, Berlin, 1985.

\bibitem{BuserBook}
Peter Buser.
\newblock Geometry and spectra of compact Riemann surfaces.
\newblock Reprint of the 1992 edition.
\newblock {\em Modern Birkh\"auser Classics}.
Birkh\"auser Boston, Ltd., Boston, MA, 2010.

\bibitem{Keen}
Linda Keen.
\newblock A rough fundamental domain for Teichmüller spaces.
\newblock {\em Bull. Amer. Math. Soc.}, 83(6):1199--1226, 1977.


\bibitem{ParkerEtAl}
John R. Parker, Norbert Peyerimhoff, and Karl Friedrich Siburg.
\newblock Minimizing length of billiard trajectories in hyperbolic polygons.
\newblock {\em Conform. Geom. Dyn.}, 22:315--332, 2018.

\bibitem{ParlierShortBers}
Hugo Parlier.
\newblock A short note on short pants.
\newblock {\em Canad. Math. Bull.}, 57(4):870--876, 2014.

\bibitem{ParlierShorterBers}
Hugo Parlier.
\newblock A shorter note on shorter pants.
\newblock {\em Bull. Lond. Math. Soc.}, 56(4):1483--1487, 2024.

\bibitem{Wolpert}
Scott A. Wolpert.
\newblock Geodesic length functions and the Nielsen problem.
\newblock {\em J. Differential Geom.}, 25(2):275--296, 1987.
\end{thebibliography}
\end{document}